\documentclass[10pt]{amsart}
\usepackage{longtable}
\newtheorem{theorem}{Theorem}[section]
\newtheorem{lemma}[theorem]{Lemma}
\newtheorem{proposition}[theorem]{Proposition}

\newcounter{claim}[theorem]

\newcounter{cclaim}[theorem]

\usepackage{color}

\newcommand{\E}{\mathrm{E}}
\newcommand{\F}{\mathrm{F}}

\newcommand{\G}{\mathrm{G}}

\newcommand{\Aut}{\mathrm{Aut}}
\newcommand{\Hom}{\mathrm{Hom}}

\newcommand{\Syl}{\mathrm{Syl}}\newcommand{\syl}{\mathrm{Syl}}

\newcommand{\GL}{\mathrm{GL}}
\newcommand{\Sp}{\mathrm{Sp}}
\newcommand{\SL}{\mathrm{SL}}
\newcommand{\0}{\emptyset}

\newcommand{\PSL}{\mathrm{PSL}}\newcommand{\PSp}{\mathrm{PSp}}
\newcommand{\Sym}{\mathrm{Sym}}
\newcommand{\Alt}{\mathrm{Alt}}

\def \eps {\epsilon}

\def \syl {\hbox {\rm Syl}}\def \Syl {\hbox {\rm Syl}}

\def \Aut{ \mathrm {Aut}}

\def \Co {\mbox {\rm Co}}

\def \PSU {\mbox {\rm PSU}}

\begin{document}

\title{A Family of fusion systems related to the groups $\Sp_4(p^a)$ and $\G_2(p^a)$}
 \author{Chris Parker}
  \author{Gernot Stroth}

\address{Chris Parker\\
School of Mathematics\\
University of Birmingham\\
Edgbaston\\
Birmingham B15 2TT\\
United Kingdom} \email{c.w.parker@bham.ac.uk}

\address{Gernot Stroth\\
Institut f\"ur Mathematik\\ Universit\"at Halle - Wittenberg\\
Theordor Lieser Str. 5\\ 06099 Halle\\ Germany}
\email{gernot.stroth@mathematik.uni-halle.de}

\maketitle

\begin{abstract}
A family  of exotic fusion systems generalizing the group fusion systems on Sylow $p$-subgroups of $\G_2(p^a)$ and $\Sp_4(p^a)$ is constructed.
\end{abstract}

\section{Introduction}
In this paper we will construct an infinite series of exotic fusion systems. More precisely for each  prime $p \geq 5$ we   build an exotic fusion system on a $p$-group which contains an extraspecial $p$-group of order $p^{p-2}$ of index $p$ (see Proposition~\ref{lem::exotic}). The catalyst for this construction  came from the authors' investigation of groups $G$ that contain a subgroup $H$ which is an automorphism group of a simple group of Lie type in characteristic $p$, such that $|G : H|$ is coprime to $p$ \cite{almostlie}.  In  \cite[Chapter 16]{almostlie} we apply the results of this article to extend the main theorem of \cite{SalStr} to groups of rank two with some exceptions related to the fact that the fusion systems constructed in this article are exotic. A thorough discussion of fusion systems is presented in \cite{Craven}.

Our construction of the exotic fusion systems develops in two phases.  First  for an arbitrary finite field $\mathbb F$, we define a group $P$ that generalizes the structure of the normalizer of a root subgroup in $\G_2(\mathbb F)$ and $\PSp_4(\mathbb F)$ and show that a certain amalgam exists if and only if $\mathbb F$ has prime order $p$ and $O_p(P)$ is extraspecial of order $p^{p-2}$. In the second phase, we show that the fusion system determined by the free amalgamated product of the amalgam is saturated and exotic.
The smallest of the amalgams is for $p=5$ on a Sylow $5$-subgroup of $\Sp_4(5)$ and has the sporadic group $\Co_1$ as a completion; however the fusion systems do not coincide.

\section{The amalgams}
Let $\mathbb F$ be a finite field of characteristic $p >0$,  $\mathbb F[X,Y]$ be the polynomial algebra in two commuting variables and $V_m$ the   $(m+1)$-dimensional subspace of $\mathbb F[X,Y]$ consisting of homogeneous polynomials of degree $p-1\ge m\ge 1$. Set  $L=\mathbb F^\times \times \GL_2(\mathbb F)$. Then, for $(t, \left(\begin{smallmatrix}\alpha &\beta\\\gamma&\delta \end{smallmatrix}\right))\in L$, we define an action of this element on $V_m$ via  $$X^aY^b \cdot(t, \left(\begin{smallmatrix}\alpha &\beta\\\gamma&\delta \end{smallmatrix}\right))=t (\alpha X + \beta Y)^a(\gamma X + \delta Y)^b$$ where $a+b=m$. Since $m< p$, $V_m$ is an irreducible $\mathbb FL$-module \cite{BN}.

 Define a bilinear function $\beta_m$ on $V_m$ by setting $$\beta_m(X^aY^b,X^cY^d)=\begin{cases} 0& \text{if } a\ne d\\\frac{(-1)^a}{\left ( m \atop a\right)}&\text{if } a=d\end{cases}$$
 and extending bilinearly.

 \begin{lemma}
 The following hold:
 \begin{enumerate}
 \item  $\beta_m$ is alternating if and only if $m$ is odd.
\item    $\beta_m$ is non-degenerate.
 \item $\beta_m$ is preserved up to scalars by $L$ and the scale factor of an element $(t, A)\in L$  is $t^2(\det A)^m$.
 \end{enumerate}
 \end{lemma}

 \begin{proof}
Since the matrix associated with $\beta_m$ has zeros everywhere other than on the anti-diagonal, $\beta_m$ is non-degenerate and alternating if and only if $m$ is odd.

Let $G=\{(1,A) \mid A \in \SL_2(\mathbb F)\}$. We will show that $\beta_m$ is $G$-invariant. For this exercise we forget the first factor of the elements of $G$ and simply work with matrices. We also suppress $\beta_m$.  It suffices to prove that the form is invariant under a  set of generators of $G$.  So observe that
 $$G= \langle \left(\begin{smallmatrix}\lambda&0\\0&\lambda^{-1} \end{smallmatrix}\right),
 \left(\begin{smallmatrix}0&1\\-1&0 \end{smallmatrix}\right),
 \left(\begin{smallmatrix}1&0\\1&1 \end{smallmatrix}\right)\mid \lambda \in \mathbb F\rangle.$$
Suppose that $\lambda \in \mathbb F$. Then
\begin{eqnarray*}
(X^aY^b\left(\begin{smallmatrix}\lambda&0\\0&\lambda^{-1} \end{smallmatrix}\right), X^cY^d\left(\begin{smallmatrix}\lambda&0\\0&\lambda^{-1} \end{smallmatrix}\right))&=&((\lambda X)^a(\lambda^{-1}Y)^b, (\lambda X)^c(\lambda^{-1}Y)^d)\\\
&=& \lambda^{a-b+c-d}(X^aY^b,X^cY^d).
\end{eqnarray*}
 Since $(X^aY^b,X^cY^d)$ is only non-zero when $a=d$ (so $b=c$), the form is invariant under these elements.
We have
 \begin{eqnarray*}
(X^aY^b \left(\begin{smallmatrix}0&1\\-1&0 \end{smallmatrix}\right), X^cY^d \left(\begin{smallmatrix}0&1\\-1&0 \end{smallmatrix}\right))&=& (Y^a(-X)^b,Y^c(-X)^d)\\
&=&(-1)^{b+d}  (X^bY^a,X^dY^c).
\end{eqnarray*}
 The last term is non-zero if and only if $b=c$ (so $a=d$). Hence the final term is
 $$(-1)^{b+d}(X^bY^a,X^dY^c) = \frac{(-1)^{b+b+d}}{ \left( m \atop b\right) }=  \frac{(-1)^a}{ \left( m \atop m-b\right) } = \frac{(-1)^a}{ \left( m \atop a\right) }= (X^aY^b,X^cY^d)$$ as required.

Finally we consider
\begin{eqnarray*}
(X^aY^b \left(\begin{smallmatrix}1&0\\1&1 \end{smallmatrix}\right), X^cY^d \left(\begin{smallmatrix}1&0\\1&1 \end{smallmatrix}\right))&=& (X^a(X+Y)^b, X^c(X+Y)^d)\\
&=& ( \sum_{j=0}^b \left (b \atop j\right)X^{a+j}Y^{b-j}, \sum_{k=0}^d \left ( d \atop k\right)X^{c+(d-k)}Y^k)\\
&=& \sum_{j=0, a+j=k}^b (-1)^{a+j}\frac{\left(b \atop j\right)\left(d \atop k\right)}{\left(m \atop a+j\right)}\\
&=& \sum_{j=0, a+j=k}^b  (-1)^{a+j}\frac{b! d! (a+j)!(m-a-j)!}{(b-j)!j!(d-k)!k! m!}\\
&=& \sum_{j=0, a+j=k}^b  (-1)^{a+j}\frac{b! d! }{j!(d-k)! m!}\\
&=& \frac{b!d!}{m^!(d-a)!}\sum_{j=0}^{d-a}  (-1)^{a+j}\frac{(d-a)!}{j!(d-a-j)!}\\
&=& \frac{b!d!}{m^!(d-a)!}\sum_{j=0}^{d-a}  (-1)^{a+j}\left( d \atop j\right).
\end{eqnarray*}

Now the final term here is zero unless $d=a$ in which case
$$\frac{b!d!}{m^!(d-a)!}\sum_{j=0}^{d-a}  (-1)^{a+j}\left( d \atop j\right)= (-1)^a\frac{(m-a)!a!}{m!} = \frac{(-1)^a}{\left( m \atop a\right)} $$
as required to show that the form is invariant.  This establishes (ii).

Given (ii), to prove (iii), we note that the matrix
$(t,\left(\begin{smallmatrix}\lambda &0\\0&1 \end{smallmatrix}\right))$
 scales the form by $t^2\det \lambda ^m$.
 \end{proof}

From now on suppose that both $p$ and  $m$ are odd. The construction of the group which will turn out to be $O_p(P)$ only requires that $\beta_m$ is a non-degenerate alternating form.
Set  $$Q= V_m\times \mathbb F^+$$  and define a binary operation on  $Q$  by $$(v, y)(w,z) = (v+w,y+z+\beta_m(v,w))$$ for $(v,y), (w,z)\in Q$.
  Then, as $\beta_m$ is alternating, $\beta_m(v,v)= 0$ and so $Q$ is a group.

 \begin{lemma}\label{lem::centralizers1} The following statements hold.
 \begin{enumerate}
 \item If $(v, y) \in Q $, then $C_Q((v,y)) = \{(w,z)\mid w\in v^\perp, z \in \mathbb F\}$.
 \item The $p$-group $Q$ is special with $$Z(Q)= \{(0,\lambda)\mid \lambda \in \mathbb F\}= Q'= \Phi(Q).$$
 \end{enumerate}
 \end{lemma}

 \begin{proof}
 Let $(w,z) \in C_Q((v,y))$. Then $$(w,z)(v,y)= (w+v,z+y+\beta_m(w,v))$$ and $$(v,y)(w,z)= (v+w,y+z+\beta_m(v,w)).$$
Since $\beta_m$ is alternating, we see that these two equation are equal if and only if $\beta_m(v,w)= 0$. Thus (i) holds and, as $\beta_m$ is non-degenerate, we have $Z(Q)= \{(0,\lambda)\mid \lambda \in \mathbb F\}$.

Plainly $Q/Z(Q)$ is abelian of exponent $p$. Hence to prove (ii), we just need to show that $Q'\ge Z(Q)$. So we calculate
\begin{eqnarray*}[(v,y),(w,z)]&=& (-v,-y)(-w,-z)(v,y)(w,z)\\&=& (-v-w,-y-z+\beta_m(y,z))(v+w,y+z+\beta_m(y,z))\\&=&(0,2 \beta(y,z)).\end{eqnarray*}
Thus (ii) follows as $p$ is odd.

 \end{proof}

 For $(t, A) \in L$ and $(v,z) \in Q$ define $$(v,z)^{ (t,A)}= (t(v\cdot A),t^2(\det A)^mz).$$ Notice that
 \begin{eqnarray*} ((v, y) (w,z)) ^{(t,A)}  &=& (v+w,y+z+\beta_m(v,w))^{(t,A)}\\ &=& (t(v+w)\cdot A, t^2(\det A)^m)(y+z+\beta_m(v,w))\\ &=& (tv\cdot A + t w \cdot A, t^2(\det A)^my+ t^2 (\det A)^mz+ \beta_m(t v\cdot A, t w \cdot A))\\
 &=& (tv \cdot A, t^2 (\det A)^my)  (tw \cdot A,t^2(\det A)^mz)\\
 &=&
 (v, y)^{(t,A)}  (w,z) ^{(t,A)}.
  \end{eqnarray*}
Therefore, $L$ acts on $Q$.

We define the following subgroups of $L$: $$B_0 = \mathbb F^\times \times \left\{\left(\begin{smallmatrix}\alpha&0\\\gamma&\beta\end{smallmatrix}\right)  \mid  \alpha, \beta \in \mathbb F^\times, \gamma \in \mathbb F\right\}$$
 and $$S_0 = \{1\} \times \left\{\left(\begin{smallmatrix}1&0\\\gamma&1\end{smallmatrix}\right)  \mid  \gamma \in \mathbb F\right\}.$$

Next we form the semidirect product of $Q$ and $L$ and some subgroups
 \begin{eqnarray*}P& =&P(m,\mathbb F)=LQ;\\
 B&=& B_0Q; \text{ and}\\
 S&=&S_0Q.
 \end{eqnarray*}

Plainly  $B= N_P(S)$.

 \begin{lemma}\label{lem::centralizers2} Suppose that $m< p$. Then the following hold:
 \begin{enumerate}
\item $C_L(Q)= \{(\mu^{-m}, \left(\begin{smallmatrix}\mu&0\\0&\mu\end{smallmatrix}\right))\mid \mu \in \mathbb F^\times\}$;
\item $C_Q(S_0)= \langle (\mu X^m,\lambda) \mid \lambda, \mu \in \mathbb F\rangle$; and
\item $C_{Q/Z(Q)}(S_0)= C_{Q/Z(Q)}(s) = C_Q(S_0)/Z(Q)$ for all $s \in S_0^\#$.
 \end{enumerate}
 \end{lemma}

 \begin{proof}  Obviously $C_L(Q) \leq Z(L)$ and so the elements are of type $(t,\left(\begin{smallmatrix}\mu&0\\0&\mu\end{smallmatrix}\right))$
 Now the action on $(X^m,0)$ gives $(t\mu^mX^m,0)$. Hence $t = \mu^{-m}$.
 \\
 \\
 As $Q/Z(Q)$ is an irreducible $L$-module, we have with \cite{Sm} that $C_{Q/Z(Q)}(S_0)$ is 1-dimensional and so the same applies for all $1 \not= s \in S_0$. Obviously $S_0$ centralizes $(X^m,0)$, which implies (iii). As $[Z(Q),S_0] = 1$, also (ii) follows. \end{proof}

 So we have constructed the group $P$ of our amalgam. Next we will construct $K$, which is an extension of the natural module by $\SL_2(\mathbb F)$.  Hence we will consider  $K$ as the subgroup of $\SL_3(\mathbb F)$  consisting of the matrices of the form $$\left(\begin{smallmatrix}1&0&0\\\alpha&\beta&\gamma\\\delta&\eps&\phi\end{smallmatrix}\right)$$ with determinant $1$.
 Let $$C=  \left\{\left(\begin{smallmatrix}1&0&0\\\alpha&\theta&0\\\delta&\eps&\theta^{-1}\end{smallmatrix}\right) \mid \alpha, \delta, \eps \in \mathbb F, \theta\in \mathbb F^\times \right\},$$
 $$D=  \left\{\left(\begin{smallmatrix}1&0&0\\\alpha&1&0\\\delta&\eps&1\end{smallmatrix}\right) \mid \alpha, \delta, \eps  \in \mathbb F\right\}$$
  and $W= O_p(K)$. So $$W= \left\{\left(\begin{smallmatrix}1&0&0\\\alpha&1&0\\\delta&0&1\end{smallmatrix}\right)\mid \alpha, \delta \in \mathbb F\right\}.$$
 Then we have \begin{eqnarray}\label{q1}
\left(\begin{smallmatrix}1&0&0\\\alpha&1&0\\\delta&\eps&1\end{smallmatrix}\right)^{\left(\begin{smallmatrix}1&0&0\\0&\theta&0\\0&0&\theta^{-1}\end{smallmatrix}\right)}= \left(\begin{smallmatrix}1&0&0\\\alpha\theta^{-1}&1&0\\\delta\theta&\eps\theta^2&1\end{smallmatrix}\right).\end{eqnarray}

Our objective is to determine under what conditions $P/C_L(Q)$ has a subgroup, which we  call $W_0$, such that $N_{P/Z(P)}(W_0Z(P)/Z(P))$ is isomorphic to $C$ in such a way that $W_0$ maps to $W$. Suppose that $W_0$ is such a subgroup, and let $C_0$ be the preimage of $N_P(W_0)$. Obviously $Z(Q) =Z(S)\le W_0$ and, after conjugation in $P$, we may assume that $N_S(W_0)\in \Syl_p(C_0)$. Suppose that $W_0 \le Q$. Then,  as $Q$ is special by Lemma~\ref{lem::centralizers1} (ii), $N_S(W_0) = Q$  and so $m=1$.  Furthermore, $W_0= \{(u,z)\mid u\in w^ \perp\}$ where $x$ is an arbitrary member of $W_0 \setminus Z(Q)$. But then $S$ normalizes $W_0$, which is impossible as $S \not \le C_0$, As $Z(Q) \le W_0$, $N_S(W_0)$ contains $R=\langle(X^m, \lambda)\mid \lambda \in \mathbb F\rangle$ which is the preimage of $C_{Q/Z(Q)}(S)$. Let $Q_0 = C_Q(R)$.  Then $ W_0Q_0/Q_0$ is normalized by $C_0Q_0/Q_0 \le B/Q_0$.  Since $C_0$ acts irreducibly on $W_0/Z(Q)$, $|W_0Q_0/Q_0|=p^a$ and, since $B$ normalizes $Q$ and $Q_0S$, $Q_0W$ is diagonal to these subgroups. We intend to determine the elements of $B$ which are candidates for the diagonal elements of $C_0$.  Now acting on the $S_0Q_0/Q_0$, elements of the form $d=(t,  \left(\begin{smallmatrix}\lambda &0\\0&\mu\end{smallmatrix}\right))$ normalize $\langle \lambda Y^m\mid \lambda \in \mathbb F\rangle+Q_0$ and $ Q_0S_0$. Furthermore,  $d$ acts by scaling $Y^m$ by $t\mu^m$ and mapping $\left(\begin{smallmatrix}   1 &0\\1&1\end{smallmatrix}\right)$ to  $\left(\begin{smallmatrix}   1 &0\\\lambda \mu^{-1}&1\end{smallmatrix}\right)$.  Assume that  $d \in C_0$ is in the image of an element of $C$ which   acts on $D$ as $\left(\begin{smallmatrix}1&0&0\\0&\theta&0\\0&0&\theta^{-1}\end{smallmatrix}\right)$. Then   the entries in $d$ must satisfy
\begin{eqnarray}\label{eq1}
t\mu^m=\lambda\mu^{-1}= \theta^{-1}.
\end{eqnarray}
Furthermore, using Equation (\ref{q1}) and $(X^m)^{d}= t\lambda^mX_m$,  we additionally require:
 \begin{eqnarray}\label{eq2}
t\lambda^m= \theta^2.
\end{eqnarray}
There is also a third equation forced by the action of $d$ on $Z(Q)$, but it turns out that this is dependent on the former two equations.
The first equality in Equation~(\ref{eq1}) yields  $t= \lambda \mu^{-m-1}$ and then combining Equation~(\ref{eq1}) and (\ref{eq2}) gives us
\begin{eqnarray*}
1&=& \theta^{-2}\theta^2=t^2\mu^{2m}t\lambda^m= t^3 \mu^{2m}\lambda^m\\&=&\lambda^3\mu^{-3m-3}\mu^{2m}\lambda^m=(\lambda\mu^{-1})^{m+3}= \theta^{-(m+3)}.
\end{eqnarray*}
Since $\theta$ is an arbitrary element of $\mathbb F^\times$, it follows that every element of $\mathbb F^\times$ is an $(m+3)$rd root of unity. Thus, $\mathbb F$ is finite and  letting $ p^a= |\mathbb F|$, we have $p^a-1$ divides $m+3$.  Since $m<p$ and $m$ is odd, we deduce that $a=1$ and either $p=3$ and $m=1$ or $p-1= m+3$.  Assume that $p=3$ and $m=1$. Let $$U= \langle  (1,\left( \begin{smallmatrix} 1&0\\\gamma&1 \end{smallmatrix}\right))(\mu X^m,0)\mid \gamma, \mu \in \mathbb F\rangle.$$ Then $U$ has index $9$ in $S$ and $U$ contains no non-trivial normal subgroups of $S$. Hence $S$ is isomorphic to a Sylow $3$-subgroup of $\Alt(9)$.  In the Sylow $3$-subgroup of $\Alt(9)$ we can show that every elementary abelian subgroup of order $9$ is contained in the extraspecial subgroup of order $27$ or has centralizer of order $27$. Hence this case does not occur.

\begin{proposition}\label{lem::amalgam}
 Suppose that $P/C_L(Q)$ has a $p$-subgroup $W_0$ such that $N_{P/Z(P)}(W_0Z(P)/Z(P))$ is isomorphic to $C$ in such a way that $W_0$ maps to $W$. Then   $\mathbb F$ has prime order $p$ and $p=m+4$.
\end{proposition}

So from now on suppose that $p= m+4$ and $\mathbb F$ has order $p$. In particular, by Lemma~\ref{lem::centralizers2}, $Q$ is extraspecial of order $p^{p-2}$ and of exponent $p$ and $P$ is isomorphic to a subgroup of $p^{1+(p-3)}_+: \Sp_{p-3}(p)$ which is isomorphic to a subgroup of $\Sp_{p-1}(p)$.  In this representation the Jordan form of a $p$-element of $S$ has no blocks of size $p$ and consequently $S$ has exponent $p$. We now explicitly show that when $p=m+4$, then $C$ is isomorphic to a subgroup of $P/C_L(Q)$. To do this, we first write down a candidate for $W_0$ and then determine its normalizer.

Define $$w(\alpha)= \sum_{j=0}^m \frac{\alpha^{j+1}}{(j+1)}\left( m \atop j\right ) X^jY^{m-j} \in V_m$$ and set $$W_0= \langle (1,\left( \begin{smallmatrix} 1&0\\\gamma&1 \end{smallmatrix}\right))(w(\gamma), \delta))\mid \gamma, \delta\in \mathbb F\rangle \le S.$$

To calculate explicitly in the extraspecial group $Q$, we need the following facts:
\begin{lemma}\label{formwawb} Suppose that $\lambda, \mu \in \mathbb F$. Then
\begin{enumerate}
\item $\beta_m(\lambda X^m , w(\mu)) = -\lambda \mu$; and
\item $\beta_m(w(\lambda),w(\mu)) = \frac{(\lambda-\mu)^{m+2}- \lambda^{m+2}+ \mu ^{m+2}}{(m+1)(m+2)} .$
\end{enumerate}
\end{lemma}

\begin{proof}
For the first part, as the coefficient of $Y^m$ in $w(\mu)$ is $\alpha$, we have
$$\beta_m(\lambda X^m , w(\mu)) = \beta_m(\lambda X^m,\mu Y^m) = - \lambda \mu.$$  For part (ii), we calculate
\begin{eqnarray*}
\beta_m(w(\lambda),w(\mu))&=& \beta_m(\sum_{j=0}^m \frac{\lambda^{j+1}}{(j+1)}\left( m \atop j\right ) X^jY^{m-j},
\sum_{j=0}^m \frac{\mu^{j+1}}{(j+1)}\left( m \atop j\right ) X^jY^{m-j})\\
&=& \sum_{j=0}^m \frac{\lambda^{j+1}}{(j+1)}\left( m \atop j\right )  \frac{\mu^{m-j+1}}{m-j+1}\left( m \atop m-j\right ) \frac{(-1)^j}{\left( m \atop j\right )}\\
&=&\sum_{j=0}^m \lambda^{j+1}(-\mu)^{m+2-(j+1)}\left( m \atop j\right )\frac {1}{(j+1)(m+2-(j+1))}\\
&=&\frac{1}{(m+1)(m+2)}\sum_{j=0}^m \lambda^{j+1}(-\mu)^{m+2-(j+1)}\left( m+2 \atop j+1\right ) \\
&=& \frac{(\lambda-\mu)^{m+2}- \lambda^{m+2}+ \mu ^{m+2}}{(m+1)(m+2)} .
\end{eqnarray*}
\end{proof}

\begin{lemma}\label{waaction}   Let $a,b,\lambda,\mu\in \mathbb F$ with $a$ and $b$ non-zero. Then
$$w(\lambda)^{(1,\left(\begin{smallmatrix}a&0\\\mu&b\end{smallmatrix}\right))}=\frac{b^{m+1}}{a}\left(w\left(\frac{a\lambda +\mu}{b}\right)- w\left(\frac{\mu}{b}\right)
\right).$$
\end{lemma}

\begin{proof}

\begin{eqnarray*}
w(\lambda)^{(1,\left(\begin{smallmatrix}a&0\\\mu&b\end{smallmatrix}\right))}&=& \left(\sum_{j=0}^m \frac{\lambda^{j+1}}{(j+1)}\left( m \atop j\right ) X^jY^{m-j}\right)^{(1,\left(\begin{smallmatrix}a&0\\\mu&b\end{smallmatrix}\right))}\\&=&  \sum_{j=0}^m \frac{\lambda^{j+1}}{(j+1)}\left( m \atop j\right ) (aX)^j(\mu X + b Y)^{m-j}\\
&=&  \sum_{j=0}^m \frac{\lambda^{j+1}}{(j+1)}\left( m \atop j\right ) (aX)^j\left( \sum_{k=0}^{m-j}\left( {m-j}\atop k\right)(\mu X)^k(bY)^{m-j-k}\right).
\end{eqnarray*}
We now determine the coefficient of $X^eY^{m-e}$:
\begin{eqnarray*}
\sum_{f=0}^e\frac{\lambda^{f+1}}{{f+1}} \left ( m \atop f \right )a^f\left({m-f}\atop{e-f}\right) \mu^{e-f}b^{m-e}&=& \sum_{f=0}^e \frac{\lambda^{f+1}a^f\mu^{e-f}b^{m-e}}{e+1}\left(m \atop e\right)\left({e+1}\atop {f+1}\right)\\
&=&\frac{b^{m-e}}{a(e+1)} \left (m \atop e\right) \sum_{f=0}^e  \lambda^{f+1} a^{f+1}\mu^{(e+1)-(f+1)}\left({e+1}\atop {f+1}\right)\\
&=& \frac{b^{m+1}}{a(e+1)} \left (m \atop e\right) \left( \left(\frac{a\lambda +\mu}{b}\right)^{e+1} - \left(\frac{\mu}{b}\right)^{e+1}\right).
\end{eqnarray*}
Therefore, $$w(\lambda)^{(1,\left(\begin{smallmatrix}a&0\\\mu&b\end{smallmatrix}\right))}=\frac{b^{m+1}}{a}\left(w\left(\frac{a\lambda +\mu}{b}\right)- w\left(\frac{\mu}{b}\right)
\right),$$ as claimed.

\end{proof}

One of the nice consequences of Lemma~\ref{waaction} is that $W_0$ is a subgroup of $S$.

By the discussion before Proposition~\ref{lem::amalgam}, we have $N_P(W_0)Q/Q $ is isomorphic to a subgroup of $$ \{(\frac{a}{b^{m+1}},\left(\begin{smallmatrix}a&0\\0&b \end{smallmatrix}\right))\mid a,b \in \mathbb F\}.$$

On the other hand, Lemma~\ref{waaction} shows that elements of the form $(\frac{a}{b^{m+1}},\left(\begin{smallmatrix}a&0\\0&b \end{smallmatrix}\right)) (0,0)$ normalize $W_0$.

\begin{lemma}\label{norm} We have
$$N_P(W_0) = \left\{(\frac{a}{b^{m+1}},\left(\begin{smallmatrix}a&0\\b\lambda&b \end{smallmatrix}\right))(w(\lambda)+ \tau X^m, \theta)\mid a, b \in \mathbb F^\times, \lambda, \tau, \theta \in \mathbb F\right\}.$$
\end{lemma}\qed

In a moment we shall write down a homomorphism from $N_P(W_0)$ onto $C$. To check that this is a homomorphism the following remark is helpful. (Note that it uses $p=m+4$.)

\begin{lemma}\label{prod}
Suppose that $$x=(\frac{a}{b^{m+1}},\left(\begin{smallmatrix}a&0\\b\lambda&b \end{smallmatrix}\right))(w(\lambda)+ \tau X^m, \theta)$$ and $$y = (\frac{c}{d^{m+1}},\left(\begin{smallmatrix}c&0\\d\mu &d \end{smallmatrix}\right))(w(\mu)+ \sigma X^m, \phi)$$ are elements of $N_P(W)$. Then \begin{eqnarray*}xy&=& (\frac{ac}{(bd)^{m+1}},\left(\begin{smallmatrix}ac&0\\cb \lambda +bd\mu &bd \end{smallmatrix}\right))(w(\frac{cb \lambda +bd\mu}{bd})+ (\frac{c^{m+1}}{d^{m+1}} \tau +\sigma) X^m, \\&&\;\;\;\;\;\; \frac{c^{m+2}}{d^{m+2}}\theta+\phi+
\left(
\frac
{(\frac{c \lambda}{d})^{m+2} - (\frac{c \lambda}{d}+\mu)^{m+2}+ \mu ^{m+2}}
{(m+1)(m+2)}\right)-\frac{c^{m+1}}{d^{m+1}}}{\tau  \mu + \frac{c\lambda \sigma}{d}).
\end{eqnarray*}
\end{lemma}\qed

Now a straightforward calculation using Lemma~\ref {prod} shows that the  map $\Theta$ defined by
$$(ab^{-m-1},\left(\begin{smallmatrix}a&0\\\lambda&b \end{smallmatrix}\right))(w(\lambda)+ \tau X^m, \theta)\mapsto
\left(\begin{smallmatrix} 1&0&0\\\frac{b}{a}\lambda&\frac{b}{a}&0\\
\frac{a}{b}(\theta + \frac{\lambda^{m+2}}{(m+1)(m+2)}-\lambda \tau)&-2\frac{a}{b} \tau&\frac{a}{b}\end{smallmatrix}\right)$$
is a surjective homomorphism from $N_P(W_0)$ to $C$ with kernel $C_L(Q)$.

Combining the above discussion with Proposition~\ref{lem::amalgam} yields
\begin{theorem}\label{thm::amalgam}  Suppose that $p$ is an odd prime and $m\le p-1$ is also odd. Set  $P= P(m,\mathbb F)= LQ$.
Then $P/C_L(Q)$ has a $p$-subgroup $W_0$ such that $N_{P/Z(P)}(W_0Z(P)/Z(P))$ is isomorphic to $C$ in such a way that $W_0$ maps to $W$ if and only if $\mathbb F$ has prime order $p$ and  $p=m+4$.
\end{theorem}\qed

Using the homomorphism $\Theta$ we can build the free amalgamated product $$G = P/C_L(Q) \ast_C K.$$

\section{The fusion systems}

Saturated fusion systems  were designed by Puig to capture the $p$-local properties of defect groups of $p$-blocks in representation theory.
 Given a group $X$ and  $p$-subgroup $T$, the  fusion system ${\mathcal F}_T(X)$  is a category with objects the subgroups of $T$ and, for objects $P$ and $Q$, the morphisms from $P$ to $Q$, $\Hom_{{\mathcal F}_T(X)}(P,Q)$, are the conjugation maps $c_x$ where $x \in X $ and $P^x \le Q$.  If $X$ is a finite group and $T \in \syl_p(X)$, then ${\mathcal F}_T(X)$ is saturated. An exotic fusion system is a saturated fusion system which is not ${\mathcal F}_T(X)$ for any finite group $X$ with $T \in \syl_p(X)$.
 For an extensive introduction to fusion systems  we recommend Craven's book \cite{Craven}.

In this section we construct a fusion system from the free amalgamated product   $G = P/C_L(Q) \ast_C K$. Set $P_1 = P/C_L(Q)$ and  identify $P_1$ and $K$ with their images in $G$, set $C = P_1 \cap K$ and $D = O_p(C)$. We identify subgroups of $S$ with their images in the quotient $P_1$.  Thus  $D=S \cap K$, $Q= O_p(P_1)$ and $W= O_p(K)$.
 We intend to show that ${\mathcal F} = {\mathcal F}_S(G)$ is an exotic  saturated fusion system. We see that ${\mathcal F}$ contains two sub-fusion systems ${\mathcal F}_S(P_1)$, ${\mathcal F}_{D}(K)$  and since $P_1$ and $K$  are finite groups with Sylow $p$-subgroups $S$ and $D$ respectively, these fusion systems are saturated.  A finite $p$-subgroup $T$ of an infinite group  $X$ is a Sylow $p$-subgroup of $X$ provided every finite $p$-subgroup of $X$ is $X$-conjugate to a subgroup of $T$.

\begin{lemma}\label{GR} We have that $S$ is a Sylow $p$-subgroup of $G$ and $${\mathcal F} = \langle {\mathcal F}_S(P_1), {\mathcal F}_{D}(K) \rangle$$  is the smallest fusion system on $S$ which contains  both ${\mathcal F}_S(P_1)$ and  ${\mathcal F}_{D}(K)$. \end{lemma}

\begin{proof} This follows with \cite[Theorem 1]{GR}.
\end{proof}

 The free amalgamated product $G$ determines a  graph  $\Gamma$. This graph has vertices all the cosets of $P_1$ in $G$ and all the cosets of $K$ in $G$. Two vertices are adjacent precisely when they have non-empty intersection. By \cite[Theorem 7]{Serre}, $\Gamma$ is a tree and, by construction, $G$ acts transitively on the edges of $\Gamma$ and has two orbits on the vertices of $\Gamma$. Moreover, $\Gamma$ is bi-partite.  We let $\alpha = P_1$ and $\beta = K$ (vertices of $\Gamma)$ and note that $(\alpha,\beta)$ is an edge. For $\gamma \in \Gamma$, $\Gamma(\gamma)$ denotes the set of neighbours of $\gamma$ in $\Gamma$.  The stabiliser $G_\gamma$ of $\gamma$ in $G$ is $G$-conjugate to either $P_1$ or $K$ and, especially, $G_\alpha= P_1$ and $G_\beta = K$.  Finally, we note that $G_\gamma$ operates transitively on $\Gamma(\gamma)$. For a subgroup $X$ of $G$  denote by $\Gamma^X$ the subgraph of $\Gamma$ fixed vertex wise by  $X$. Since $\Gamma$ is a tree, so is $\Gamma^X$.

In the proof of part (iii) of the next lemma, we need to consider centric subgroups of $\mathcal F$. These are subgroups $T$ of $S$ such that $C_S(T\alpha)= Z(T\alpha)$ for all $\alpha \in \Hom_{\mathcal F_S(G)}(T,S)$. Note that centric subgroups have order at least $p^2$.

\begin{lemma}\label{sat} If $X \le S$ has order at least $p^2$, then $\Gamma^X$ is finite. Furthermore,
\begin{enumerate}
\item $X$ does not fix a path of length $5$ with middle vertex  a coset of $P_1$.
\item $N_G(X)$ is conjugate to a subgroup of either $P_1$ or $K$; and
\item ${\mathcal F} $  is a saturated fusion system.
\end{enumerate}
\end{lemma}

\begin{proof}  Assume that $X$ is a $p$-subgroup of $G$ of order at least $p^2$.
Assume the path $\pi=(\alpha_1,\alpha_2,\alpha_3,\alpha_4, \alpha_5)$ of length $5$ is in $\Gamma^X$ with $\alpha_3$ a coset of $P_1$.  Then, as $G$ acts edge transitively on $\Gamma$, we may conjugate $X$ and the path $\pi$ so that $\alpha_2= \beta$ and $\alpha_3=\alpha$.  Then $X \le G_\alpha \cap G_\beta = P_1 \cap K= C$. Thus $X \le D$ as $X$ is a $p$-group.  Now $X$ also fixes $\alpha_1$  and so, as any two Sylow $p$-subgroups of $K$ intersect in $W$ and $X$ has order at least $p^2$, we have that $W= X$.  Similarly, $X= O_p(G_{\alpha_4})$ and consequently $X= W=O_p(G_\beta)= O_p(G_{\alpha_4})$.  Now $G_\alpha$ acts transitively on $\Gamma(\alpha)$ and so there exists $g \in G$ so that $G_{\alpha_4}^g = G_\beta$.  But then $g$ normalizes $W$. Since, by Lemma~\ref{norm},  $G_\beta =K\ge C=N_{P_1}(W)=N_{G_{\alpha}}(W)$, we see that $\beta= \alpha_4$ and this is a contradiction. This proves (i).  Using (i), we see that $X$ fixes no paths of length $6$ and so $|\Gamma^X|$ is finite.  In particular,  if $X$ is a centric subgroup in ${\mathcal F}$  then $|\Gamma^X|$ is finite.

Since $\Gamma^X$ is finite and $\Gamma$ is bipartite, we now see with \cite[Corollary, page 20]{Serre} that $N_G(X)$ fixes a vertex of $\Gamma$.  This proves (ii).

Finally,  application of \cite[Corollary 3.4]{CP} yields that ${\mathcal F}$ is saturated and this is (iii).
\end{proof}

A subgroup $X$ of $S$ is  fully $\mathcal F$-normalized provided $|N_S(X)|\ge |N_S(X\alpha)|$ for all $\alpha \in \Hom_{\mathcal F}(X,S)$.

\begin{lemma}\label{wcl}
We have
  $W$ is centric and fully $\mathcal F$-normalized.
\end{lemma}

\begin{proof} Suppose that $\alpha \in \Hom_\mathcal F(W,S)$ and $|N_S(W\alpha)| > |N_S(W)|$. We have that $\alpha = c_g$ for some $g \in G$ with $W^g \le S$.
 Since $|N_S(W^g)| > p^3$,  $P_1$ is the unique vertex of $\Gamma$ fixed by  $N_S(W^g)$. Since $W^g$ fixes $Kg$ and $Kgh$ for all $h \in N_S(W^g)$, we infer that $W^g$ fixes a path of length at least $5$ with $P$ at the middle vertex, this contradicts Lemma~\ref{sat}(i). Hence $W$ is fully $\mathcal F$-normalized. Since $|N_S(W)|=p^3$, it also follows that $W$ is centric.
\end{proof}

\begin{lemma} The fusion system $\mathcal F={\mathcal F}_S(G)$ is exotic.
\end{lemma}

\begin{proof} Suppose ${\mathcal F}$ is not exotic and let $H$ be a finite group with Sylow $p$-subgroup $S$ such that  ${\mathcal F}_S(H)={\mathcal F}_S(G)$.
Then there exists subgroups $Q_0$ and $W_0$ of $S$ such that $
\Aut_H(Q_0) = \Aut_{\mathcal F}(Q)$ and $\Aut_H(W_0) = \Aut_\mathcal F(W)$. Let $K_0 = N_H(W_0)$ and $P_0= N_{H}(Q_0)$.

We may assume $O_{p^\prime}(H) = 1$. Let $N$ be a minimal normal subgroup of $H$. Then, as $p$ divides $|N|$ and $Z(S)$ has order $p$,   $Z(S) \leq N$. The action of ${K_0}$ implies $W= \langle Z(S)^{K_0} \rangle  \leq N$. Hence $Z_2(S)\le [Q,S] \le N$ and finally $S= QW=\langle Z_2(S)^{P_1}\rangle W \le N$. Hence $N =O^{p'}(H)$. Since $\langle S^{P_1} \rangle$ is not a $p$-group, $N$ is a direct product of isomorphic non-abelian simple groups. Moreover, as $|Z(S)|=p$, $N$ is a simple group and $C_H(N)=1$. Therefore $H$ is an almost simple group. We now consider the finite simple groups as given by the classification theorem.

Recall that $p \geq 5$, $|S|=p^{p-1}$ and $S$ is of exponent $p$ and is not abelian. In particular, $H$ is not an alternating group.

 Suppose that $H$ is  a  Lie type group in characteristic $p$. Then, by the Borel -Tits Theorem,  $K$ is contained in some parabolic subgroup $L$ of $H$ and $W_0\le O_p(K_0) \leq O_p(L)$.  As $W_0$ is centric by Lemma~\ref{wcl} (i),  $W_0 \ge Z(O_p(L))$. Since $W_0$ is fully $\mathcal F$-normalized, $W_0$ is not normal in $L$ and so $Z(O_p(L))<  W_0$, contrary to $K_0 \le L$ and $W_0$ being a minimal normal subgroup of $K_0$. Hence $H$ is not of Lie type in characteristic $p$

 Assume that $H$ is  of Lie type in characteristic $r \not= p$. If  $p$ does not divided $|Z(\hat{H})|$, where $\hat{H}$ is the universal version of $H$, then application of \cite[Theorem 4.10.3(e)]{GLS} to $W_0$ shows $p = 5$ and $H = \E_8(r^a)$. Since $|S|= 5^4$ and $5^5$ divides the order of $\E_8(r^a)$, this is impossible.  So we may assume that $p$ divides   $|Z(\hat{H})|$. In particular $H \cong \PSL_n(r^a)$ or $\PSU_n(r^a)$ and $p$ divides $(r^a-1,n)$ in the first case and $(r^a+1,n)$ in the second.   Therefore $S$ contains a toral subgroup of order $p^{n-2} \ge p^{p-2}$, as $p$ divides $n$. Since $Q_0$ is extraspecial of order $p^{p-2}$ and $|S| = p^{p-1}$, and the largest abelian subgroup of $Q_0$ is of order $p^{(p-1)/2}$, we infer that $(p-1)/2 +1 \geq p-2$, and then $p=5$. Thus $H$ contains a subgroup $L\cong 5^3.\Sym(5)$.  Since this does not embed in $P_1$, this is impossible.

Finally suppose that $H$ is a sporadic simple group.  Inspection of the lists in \cite[Table 5.3]{GLS} shows the either  $H = \Co_1$ and $p = 5$ or $H = \F_1$ and $p = 7$. In the first case we again observe a subgroup $5^3.\mathrm{PGO}_3(5)$, which does not exist in $\mathcal F$.  In the second case  $\Aut_H(Q) \cong 6. \Alt(7)$ which is impossible.

We conclude that $\mathcal F$ is exotic.
\end{proof}

We summarize the above result in a way that we can easily cite it in \cite{almostlie}.

\begin{proposition}\label{lem::exotic} Let $p \geq 5$ be a prime. Set  $P_1 = QL$, where $Q$ is extraspecial of order $p^{p-2}$, $L \cong \GL_2(p)$ and $L'$ induces the irreducible module of homogeneous polynomials in $X,Y$ of degree $p-4$ on $Q/Z(Q)$. Then for $S$ a Sylow $p$-subgroup of $QL$, there is an exotic fusion system on $S$ which extends $\mathcal F_{QL}(S)$.
\end{proposition}

We close the paper by remarking that (when $m=p-4$) the amalgam $B/C_L(Q)*_CK$ also provides an exotic fusion system. This example has exactly one class of essential subgroups.

\end{document}